\documentclass[a4paper,11pt]{article}
\usepackage[margin=0.9in]{geometry}
\usepackage{enumitem}

\usepackage{braket,hyperref}
\usepackage{amssymb,mathtools,amsthm,amsmath}
\usepackage{tikz}

\usepackage{xargs}

\usepackage{verbatim}

\theoremstyle{plain}
\newtheorem{theorem}{\bf Theorem}

\newtheorem{proposition}[theorem]{\bf Proposition}
\newtheorem{corollary}[theorem]{\bf Corollary}
\newtheorem{lemma}[theorem]{\bf Lemma}
\theoremstyle{definition}

\newenvironment{remark}[1][Remark.]{\begin{trivlist}
		\item[\hskip \labelsep {\bfseries #1}]}{\end{trivlist}}

\numberwithin{theorem}{section}
\numberwithin{equation}{section}

\newcommand{\Rea}{{\mathbb R}}

\DeclareMathOperator{\lk}{lk}

\DeclareMathOperator{\mi}{\mathcal{I}}
\DeclareMathOperator{\mj}{\mathcal{J}}

\newcommand{\lapw}[4]{L_{#1}^{#4}(#2;#3)}

\newcommand{\simlapw}[4]{\mathcal{L}_{#1}^{#4}(#2;#3)}
\newcommand{\simlapws}[4]{\hat{\mathcal{L}}_{#1}^{#4}(#2;#3)}

\newcommand{\eigsmallest}[2]{\lambda_{#1}^{\uparrow}\left(#2\right)}
\newcommand{\eiglargest}[2]
{\lambda_{#1}^{\downarrow}\left(#2\right)}

\DeclareMathOperator{\T}{\top}

\newcommand{\mats}[2]{\Rea^{#1\times #2}}
\newcommand{\matsi}[1]{\Rea^{#1\times #1}}

\newcommand{\Xv}[3]{#1_{#2}(#3)}
\newcommand{\wmindeg}[3]{\delta_{#1}(#2;#3)}
\newcommand{\wmaxdeg}[3]{\Delta_{#1}(#2;#3)}

\usepackage{authblk}

\begin{document}

    \title{An eigenvalue interlacing approach to Garland's method}

     \author{Alan Lew\thanks{Dept. Math. Sciences, Carnegie Mellon University, Pittsburgh, PA 15213, USA; 
    Department of Mathematics, Technion, Haifa 32000, Israel. \quad e-mail: \href{mailto:alanlew@technion.ac.il}{alanlew@technion.ac.il}.}}

	\date{}
	\maketitle

\begin{abstract}

Let $X$ be a pure $d$-dimensional simplicial complex. For $0\le k\le d$, let $X(k)$ be the set of $k$-dimensional faces of $X$, let $\tilde{L}_k(X)$ be the $k$-dimensional weighted total Laplacian operator on $X$, and let $\tilde{H}_k(X;\mathbb{R})$ be its $k$-dimensional reduced homology group with real coefficients. For $\sigma\in X$, let $\text{lk}(X,\sigma)$ be the link of $\sigma$ in $X$. For a matrix $M$, we denote by $\text{Spec}(M)$ the multi-set containing all the eigenvalues of $M$. We show that, for every $0\le \ell<k \le d$, 
\[
    \dim(\tilde{H}_k(X;\mathbb{R}))\le \sum_{\eta\in X(\ell)}\left| \left\{ \lambda\in \text{Spec}(\tilde{L}_{k-\ell-1}(\text{lk}(X,\eta))) :\, \lambda\le \frac{(\ell+1)(d-k)}{k+1}\right\}\right|.
\]
This extends the classical vanishing theorem of Garland, corresponding to the special case when the right hand side of the inequality is equal to zero, and a more recent result by Hino and Kanazawa, corresponding to the case $\ell=k-1$. 
A main new ingredient in our proof is an abstract version of Garland's local to global principle, which follows as a simple consequence of the eigenvalue interlacing theorem, and may be of independent interest.
\end{abstract}

\section{Introduction}

For finite sets $\mi,\mj$, we denote by $\mats{\mi}{\mj}$ the set of real valued $|\mi|\times |\mj|$ matrices with rows indexed by the elements of $\mi$ and  columns indexed by the elements of $\mj$. For a diagonalizable matrix $M\in \matsi{\mi}$ with real eigenvalues and $1\le i\le |\mi|$, we denote by $\eigsmallest{i}{M}$ the $i$-th smallest eigenvalue of $M$, and by $\eiglargest{i}{M}$ its $i$-th largest eigenvalue. We denote by $\text{Spec}(M)$ the multi-set consisting of all eigenvalues of $M$ (each one repeated according to its multiplicity). For finite sets $\mi_1,\ldots,\mi_k$ and $M_i\in \matsi{\mi_i}$ for $1\le i\le k$, let $M=\oplus_{i=1}^k M_{i}$ be the direct sum of $M_1,\ldots,M_k$. That is, $M$ is a block diagonal matrix with blocks $M_1,\ldots, M_k$. Note that $\text{Spec}(M)$ is the union of  $\text{Spec}(M_i)$, for $1\le i\le k$.

Let $X$ be a simplicial complex on a finite vertex set $V$. We fix an arbitrary linear order $<$ on $V$.  For $-1\le k\le \dim(X)$, we denote by $X(k)$ the family of $k$-dimensional simplices of $X$.  For $\sigma\in X(k-1)$ and $\tau\in X(k)$ satisfying $\sigma\subset \tau$, we define $(\tau:\sigma)=(-1)^{|\{v\in \tau:\, v<u\}|}$, where $u$ is the unique vertex in $\tau\setminus\sigma$. For $0\le k\le \dim(X)$, the \emph{$k$-th boundary matrix} of $X$ is the matrix $\partial_k(X)\in \Rea^{X(k-1)\times X(k)}$ defined by
\[
    \partial_k(X)_{\sigma,\tau}=\begin{cases}
        (\tau:\sigma) & \text{if } \sigma\subset\tau,\\
        0 & \text{otherwise,}
    \end{cases}
\]
for all $\sigma\in X(k-1)$ and $\tau\in X(k)$.

We say that $X$ is a \emph{pure $d$-dimensional simplicial complex} if all the maximal faces of $X$ are of dimension exactly $d$. For a pure $d$-dimensional simplicial complex $X$, we define a weight function $w:X\to \Rea_{>0}$ by
\[
    w(\sigma) = \left|\left\{ \tau\in X(d):\, \sigma\subset\tau\right\}\right|
\]
for all $\sigma\in X$. For $-1\le k\le d$, let $W_k(X)\in \Rea^{X(k)\times X(k)}$ be the diagonal matrix defined by $W_k(X)_{\sigma,\sigma}=w(\sigma)$ for all $\sigma\in X(k)$. For $0\le k\le d-1$, the \emph{$k$-dimensional weighted total Laplacian operator} on $X$ is the matrix $\tilde{L}_k(X)\in \Rea^{X(k)\times X(k)}$ defined by
\[
    \tilde{L}_k(X)= W_k(X)^{-1} \partial_{k+1}(X) W_{k+1}(X)\partial_{k+1}(X)^{\T} + \partial_k(X)^{\T} W_{k-1}(X)^{-1} \partial_k(X) W_{k}(X).
\]
Similarly, we define $ \tilde{L}_d(X)= \partial_d(X)^{\T} W_{d-1}(X)^{-1} \partial_d(X) W_{d}(X)\in \matsi{X(d)}$.

It is easy to check that, for all $0\le k\le d$,  $\tilde{L}_k(X)$ is diagonalizable, and all its eigenvalues are non-negative real numbers (see Section \ref{sec:higher_laplacians} for more details).  
In \cite{eckmann1945haromic}, Eckmann observed that the $k$-dimensional reduced homology group with real coefficients of $X$, denoted by $\tilde{H}_k(X;\Rea)$, is isomorphic to the kernel of $\tilde{L}_k(X)$. 

For a simplex $\sigma\in X$, let $\lk(X,\sigma)=\{\tau\in X:\, \tau\cap \sigma=\emptyset,\, \tau\cup \sigma\in X\}$ be the \emph{link} of $\sigma$ in $X$. Note that $\lk(X,\sigma)$ is a subcomplex of $X$.
In his seminal work \cite{garland1973padic}, Garland established the following relation between the spectral gap (that is, the smallest eigenvalue) of the $k$-dimensional Laplacian of a pure simplicial complex and the spectral gaps of lower dimensional Laplacian operators on links of its simplices.

\begin{theorem}[Garland \cite{garland1973padic}; see also Papikian {\cite[Corollary 3.15]{papikian2016garland}}]
\label{thm:garland}
       Let $X$ be a pure $d$-dimensional simplicial complex, and let $0\le \ell<k\le d$. Then, 
      \[
       (k-\ell)\cdot \eigsmallest{1}{\tilde{L}_k(X)} \ge (k+1)\cdot \min_{\eta\in X(\ell)} \eigsmallest{1}{\tilde{L}_{k-\ell-1}(\lk(X,\eta))} - (\ell+1)(d-k).
    \]    
\end{theorem}

The next result, sometimes referred to as ``Garland's vanishing theorem", follows immediately from Theorem \ref{thm:garland}.

\begin{corollary}[Garland \cite{garland1973padic}]\label{cor:garland}
 Let $X$ be a pure $d$-dimensional simplicial complex, and let $0\le \ell<k\le d$. If 
 \[
 \eigsmallest{1}{\tilde{L}_{k-\ell-1}(\lk(X,\eta))} > \frac{(\ell+1)(d-k)}{k+1}
 \]
 for all $\eta\in X(\ell)$, then $\tilde{H}_k(X)=0$.
\end{corollary}

Since its introduction, many extensions and variants of Garland's technique were developed, with applications to diverse fields such as group theory, matching theory, the theory of random simplicial complexes, and the study of convergence of random walks (see, for example, \cite{ballmann1997l2,zuk1996propriete,aharoni2005eigenvalues, gundert2016eigenvalues, hoffman2021spectral, kahle2014sharp,  oppenheim2018local, kaufman2020high, anari2024log,Abdolazimi2021matrix}). 
In this paper, we prove the following generalization of Theorem \ref{thm:garland}.

\begin{theorem}\label{thm:garland_plus}
    Let $X$ be a pure $d$-dimensional simplicial complex, and let $0\le \ell<k\le d$. Then, for all $1\le i\le |X(k)|$, 
      \[
       (k-\ell)\cdot \eigsmallest{i}{\tilde{L}_k(X)} \ge (k+1)\cdot \eigsmallest{i}{\bigoplus_{\eta\in X(\ell)} \tilde{L}_{k-\ell-1}(\lk(X,\eta))} - (\ell+1)(d-k).
    \]
\end{theorem}

As a consequence, we obtain the following extension of Corollary \ref{cor:garland}.

\begin{corollary}\label{cor:garland_plus}
     Let $X$ be a pure $d$-dimensional simplicial complex, and let $0\le \ell<k\le d$. Then,
\[
   \dim(\tilde{H}_k(X;\Rea)) \le \sum_{\eta\in X(\ell)} \left| \left\{ \lambda\in {\rm Spec}(\tilde{L}_{k-\ell-1}(\lk(X,\eta))):\, \lambda\le  \frac{(\ell+1)(d-k)}{k+1}\right\}\right|.
\]
\end{corollary}

The case $\ell=k-1$ of Corollary \ref{cor:garland_plus} was proved by Hino and Kanazawa in \cite[Theorem 2.5]{hino2019asymptotic}, using different methods. A main new ingredient in our proof is the following ``local to global principle" for matrices, which is a simple consequence of the eigenvalue interlacing theorem.

\begin{lemma}\label{lemma:abstract_garland}
    Let $\mi$ be a finite set, and let $\mi_1,\ldots, \mi_m\subset \mi$. For $1\leq i\leq m$, let $M_i\in \matsi{\mi_i}$ be a symmetric matrix  and let $s_i: \mi_i\to \Rea$ such that $\sum_{i:\, \sigma\in \mi_i} s_i(\sigma)^2=1$ for all $\sigma\in\mi$.
    Define $M\in \matsi{\mi}$ by
    \[
        M_{\sigma,\tau} = \sum_{\substack{i\in\{1,\ldots,m\}:\\\sigma,\tau\in \mi_i}} s_i(\sigma) s_i(\tau) (M_i)_{\sigma,\tau}     
    \]
    for all $\sigma,\tau\in \mi$. Then, for all $1\leq k\leq |\mi|$,
    \[
     \eigsmallest{k}{M} \geq  \eigsmallest{k}{\bigoplus_{i=1}^m M_i},
    \]
    and
    \[
     \eiglargest{k}{M} \leq  \eiglargest{k}{\bigoplus_{i=1}^m M_i}.
    \]
\end{lemma}

\begin{remark}
In recent work \cite{babson2023homological}, Babson and Welker presented a new variant of Garland's vanishing theorem, which extends beyond the setting of simplicial complexes to a broader family of posets (including, in particular, cubical complexes). A key ingredient in their proof is the observation that the chain complex of a simplicial complex can be obtained as the projection of the direct sum of certain chain complexes associated to the links of its simplices. This parallels, in a sense, the situation in Lemma \ref{lemma:abstract_garland}, and served as a motivation for our work.
Let us also mention our recent work \cite{lew2024laplacian}, where we used eigenvalue interlacing in a different way to extend a ``global" version of Garland's method due to Aharoni, Berger, and Meshulam \cite{aharoni2005eigenvalues}.
\end{remark}

This paper is organized as follows. In Section \ref{sec:prel}, we present background material on matrix eigenvalues and on high dimensional Laplacian operators, which we will later use. Section \ref{sec:abstract_garland} contains the proof of Lemma \ref{lemma:abstract_garland}.  In Section \ref{sec:main}, we present the proofs of our main results, Theorem \ref{thm:garland_plus} and Corollary \ref{cor:garland_plus}. In fact, we first state and prove more general versions of these results, valid for simplicial complexes with an arbitrary weight function, and then, in Sections \ref{sec:main:unweighted} and \ref{sec:main:weighted}, we show how they specialize to the cases of unweighted Laplacians, and of weighted Laplacians on pure simplicial complexes, respectively. In addition, in Section \ref{sec:main:upper}, we present an upper bound analogous to the lower bound in Theorem \ref{thm:garland_plus} (Proposition \ref{prop:garland_upper}), which follows by the same arguments.

\section{Preliminaries}\label{sec:prel}

\subsection{Matrix eigenvalues}

Recall that for finite sets $\mi,\mj$, we denote by $\mats{\mi}{\mj}$ the set of real valued $|\mi|\times |\mj|$ matrices with rows indexed by the elements of $\mi$ and  columns indexed by the elements of $\mj$. For a diagonalizable matrix $M\in \matsi{\mi}$ with real eigenvalues and $1\le i\le |\mi|$, we denote by $\eigsmallest{i}{M}$ the $i$-th smallest eigenvalue of $M$, and by $\eiglargest{i}{M}$ its $i$-th largest eigenvalue. We denote the $n\times n$ identity matrix by $I_n$.
We will need the following well-known generalization of Cauchy's interlacing theorem.

\begin{theorem}[See, for example, {\cite[Theorem 2.5.1]{brouwer2012book}}]
\label{thm:cauchy_plus}
Let $\mi,\mj$ be finite sets satisfying $|\mi|\leq |\mj|$.
Let $A\in \matsi{\mj}$ be a symmetric matrix and let $S\in \mats{\mj}{\mi}$ such that $S^{\T} S =I_{|\mi|}$. Let $B=S^{\T} A S$. Then, for all $1\leq i\leq |\mi|$,
\[
    \eigsmallest{i}{B} \geq \eigsmallest{i}{A}
\]
and
\[
    \eiglargest{i}{B} \leq \eiglargest{i}{A}.
\]
\end{theorem}

We will also use the following inequalities due to Weyl.
\begin{lemma}[See, for example, {\cite[Thm 2.8.1]{brouwer2012book}}]\label{lemma:weyl}
Let $A, B$ be real symmetric matrices of size $n\times n$. Then, for all $1\leq i\leq n$,
\[
   \eigsmallest{i}{A+B}\geq \eigsmallest{i}{A}+\eigsmallest{1}{B},
\]
and
\[
 \eiglargest{i}{A+B} \leq \eiglargest{i}{A}+\eiglargest{1}{B}.
\]
\end{lemma}

\subsection{High dimensional Laplacians}\label{sec:higher_laplacians}

Let $V$ be a finite set. A simplicial complex $X$ on vertex set $V$ is a family of subsets of $V$ closed with respect to inclusion. That is, if $\tau\in X$ and $\sigma\subset \tau$, then $\sigma\in X$. We call an element $\sigma\in X$ a \emph{face} or \emph{simplex} of $X$. The dimension of a simplex $\sigma$ is defined as $\dim(\sigma)=|\sigma|-1$. Note that the empty set is a $(-1)$-dimensional simplex of $X$. 
The dimension of $X$, denoted by $\dim(X)$, is the maximal dimension of a face of $X$. For $-1\le k\le \dim(X)$, we denote by $X(k)$ the set of $k$-dimensional simplices of $X$. 

Fix an arbitrary order $<$ on $V$, and let $0\le k\le \dim(X)$.  Recall that for $\sigma\in X(k-1)$ and $\tau\in X(k)$ such that $\sigma\subset \tau$, we defined $(\tau:\sigma)=(-1)^{|\{v\in \tau:\, v<u\}|}$, where $u$ is the unique element in $\tau\setminus\sigma$. 
We will need the following simple lemma.
\begin{lemma}[See, for example, {\cite[Lemma 2.4]{lew2025sums}}]
\label{lemma:sign}
   Let $X$ be a simplicial complex and let $0\le k\le \dim(X)-1$. Then, for $\sigma,\tau\in X(k)$ such that $|\sigma\cap \tau|=k$ and $\sigma\cup\tau\in X$,
    \[
        (\sigma\cup\tau:\sigma)(\sigma\cup\tau:\tau)=-(\sigma:\sigma\cap \tau)(\tau:\sigma\cap \tau).
    \]
\end{lemma}
Recall that we defined the $k$-th boundary matrix $\partial_k(X)\in \mats{X(k-1)}{X(k)}$ by
\[
    \partial_k(X)_{\sigma,\tau}= \begin{cases}
        (\tau:\sigma) & \text{if } \sigma\subset \tau,\\
        0 & \text{otherwise,}
    \end{cases}
\]
for all $\sigma\in X(k-1)$ and $\tau\in X(k)$. It is a well-known fact that $\partial_k(X)\partial_{k+1}(X)=0$. The \emph{$k$-dimensional reduced homology group} (with real coefficients) of $X$ is defined as the quotient
\[
    \tilde{H}_k(X;\Rea) = \frac{\text{Ker } \partial_k(X)}{\text{Im } \partial_{k+1}(X)}.
\]
For a weight function $w:X\to\Rea_{>0}$, let $W_k(X)\in \matsi{X(k)}$ be the diagonal matrix defined by
\[
    W_k(X)_{\sigma,\sigma}= w(\sigma)
\]
for all $\sigma\in X(k)$. 

Let $0\le k\le \dim(X)-1$. The \emph{$k$-dimensional $w$-weighted upper Laplacian operator} on $X$ is the matrix
\[
    \lapw{k}{X}{w}{+} = W_k(X)^{-1} \partial_{k+1}(X) W_{k+1}(X)\partial_{k+1}(X)^{\T}\in \Rea^{X(k)\times X(k)}.
\]
For convenience, for $k=\dim(X)$, we define $\lapw{k}{X}{w}{+}=0\in \matsi{X(k)}$.
For $0\le k\le \dim(X)$, the \emph{$k$-dimensional $w$-weighted lower Laplacian operator} on $X$ is the matrix
\[
    \lapw{k}{X}{w}{-} = \partial_k(X)^{\T} W_{k-1}(X)^{-1} \partial_k(X) W_{k}(X)\in \Rea^{X(k)\times X(k)}.
\]
We define the \emph{$k$-dimensional $w$-weighted total Laplacian operator} on $X$ as
\[
    \lapw{k}{X}{w}{} = \lapw{k}{X}{w}{+}+\lapw{k}{X}{w}{-}.
\]
Note that $\lapw{k}{X}{w}{+}$, $\lapw{k}{X}{w}{-}$, and $\lapw{k}{X}{w}{}$ are not in general symmetric matrices. However, we may define symmetric positive semi-definite matrices
\[
    \simlapw{k}{X}{w}{+} = W_k(X)^{1/2} \lapw{k}{X}{w}{+} W_k(X)^{-1/2}=
    W_{k}(X)^{-1/2} \partial_{k+1}(X) W_{k+1}(X)\partial_{k+1}(X)^{\T} W_{k}(X)^{-1/2},
\]
\[
  \simlapw{k}{X}{w}{-}= W_k(X)^{1/2} \lapw{k}{X}{w}{-} W_k(X)^{-1/2}= W_k(X)^{1/2}\partial_k(X)^{\T} W_{k-1}(X)^{-1} \partial_k(X) W_{k}(X)^{1/2},
\]
and
\[
\simlapw{k}{X}{w}{}=W_k(X)^{1/2}\lapw{k}{X}{w}{} W_k(X)^{-1/2}=\simlapw{k}{X}{w}{+}+\simlapw{k}{X}{w}{-},
\]
which are similar to (and thus have the same spectrum as) the matrices $\lapw{k}{X}{w}{+}$, $\lapw{k}{X}{w}{-}$, and $\lapw{k}{X}{w}{}$, respectively. In particular, $\lapw{k}{X}{w}{+}$, $\lapw{k}{X}{w}{-}$, and $\lapw{k}{X}{w}{}$ are diagonalizable, and all of their eigenvalues are real and non-negative.

For a simplex $\sigma\in X$, let $N_X(\sigma)= \{v\in V\setminus\sigma:\, \sigma\cup\{v\}\in X\}$. That is, $N_X(\sigma)$ is the vertex set of the link of $\sigma$ in $X$. 
The following explicit formulas for the matrices $\simlapw{k}{X}{w}{+}$ and $\simlapw{k}{X}{w}{-}$ follow by simple computation (where, in the case of $\simlapw{k}{X}{w}{+}$, Lemma \ref{lemma:sign} is applied). See, for example, \cite{horak2013spectra}, for similar computations.
\begin{lemma}\label{lemma:lap_formulas}
Let $X$ be a simplicial complex, let $w:X\to\Rea_{>0}$, and let $0\le k\le \dim(X)$. Then,
\begin{equation}
    \label{eq:upper_sim_laplacian_def}
   \simlapw{k}{X}{w}{+}_{\sigma,\tau}= \begin{cases}
       \sum_{v\in N_X(\sigma)}\frac{w(\sigma\cup\{v\})}{w(\sigma)} & \text{if } \sigma=\tau,\\
      -(\sigma:\sigma\cap \tau)(\tau:\sigma\cap\tau) \frac{w(\sigma\cup\tau)}{\sqrt{w(\sigma)w(\tau)}} & \text{if } |\sigma\cap \tau|=k,\, \sigma\cup\tau\in X(k+1),\\
       0 & \text{otherwise,}
   \end{cases}
\end{equation}
and
\begin{equation}
    \label{eq:lower_sim_laplacian_def}
   \simlapw{k}{X}{w}{-}_{\sigma,\tau}= \begin{cases}
       \sum_{v\in \sigma}\frac{w(\sigma)}{w(\sigma\setminus\{v\})} & \text{if } \sigma=\tau,\\
      (\sigma:\sigma\cap \tau)(\tau:\sigma\cap\tau) \frac{\sqrt{w(\sigma)w(\tau)}}{w(\sigma\cap \tau)} & \text{if } |\sigma\cap \tau|=k,\\
       0 & \text{otherwise,}
   \end{cases}
\end{equation}
for all $\sigma,\tau\in X(k)$. 
\end{lemma}
Eckmann showed in \cite{eckmann1945haromic} that, for all $k\ge 0$, $\tilde{H}_k(X;\Rea)\cong \text{Ker }\lapw{k}{X}{w}{}$. Equivalently, 

\begin{lemma}[Eckmann \cite{eckmann1945haromic}]\label{lemma:eckmann}
    Let $X$ be a simplicial complex, $w:X\to\Rea_{>0}$, and $0\le k\le \dim(X)$. Then, for $0\le j\le |X(k)|-1$,  
    $\dim(\tilde{H}_k(X;\Rea))\le j$ if and only if $\eigsmallest{j+1}{\lapw{k}{X}{w}{}}>0$.
\end{lemma}

\section{A local to global principle for matrices}\label{sec:abstract_garland}

In this section, we prove Lemma \ref{lemma:abstract_garland}, which may be seen as an abstract version of Garland's method. 
Let $\mi_1,\ldots,\mi_m$ be finite sets, and let $M_i\in \matsi{\mi_i}$ for $1\le i\le m$. Recall that the \emph{direct sum} of $M_1,\ldots,M_m$ is the block diagonal matrix 
\[
    \bigoplus_{i=1}^m M_i = \begin{pmatrix}
M_1 \\
& \ddots \\
& & M_m
\end{pmatrix}.
\]
Note that the spectrum of $\bigoplus_{i=1}^m M_i$ is the union of the spectra of $M_1,\ldots,M_m$.

\begin{proof}[Proof of Lemma \ref{lemma:abstract_garland}]

 Let $\mi$ be a finite set, and let $\mi_1,\ldots, \mi_m\subset \mi$. For $1\leq i\leq m$, let $M_i\in \matsi{\mi_i}$ be a symmetric matrix  and let $s_i: \mi_i\to \Rea$ such that $\sum_{i:\, \sigma\in \mi_i} s_i(\sigma)^2=1$ for all $\sigma\in\mi$.
    The matrix $M\in \matsi{\mi}$ is defined by
    \[
        M_{\sigma,\tau} = \sum_{\substack{i\in\{1,\ldots,m\}:\\\sigma,\tau\in \mi_i}} s_i(\sigma) s_i(\tau) (M_i)_{\sigma,\tau}     
    \]
    for all $\sigma,\tau\in \mi$.
    
    Let $\mj=\left\{ (i,\sigma):\, 1\leq i\leq m,\, \sigma\in\mi_i\right\}$. We can think of $\bigoplus_{i=1}^m M_i$ as a matrix in $\matsi{\mj}$ defined by
    \[
    \left(\bigoplus_{i=1}^m M_i\right)_{(j,\sigma),(k,\tau)}= \begin{cases}
                    (M_j)_{\sigma,\tau} & \text{if } j=k,\\
                    0 & \text{otherwise,}
    \end{cases} 
    \]
    for all $(j,\sigma),(k,\tau)\in \mathcal{J}$. 
    Let $S\in \mats{\mj}{\mi}$ be defined by
    \[
        S_{(i,\sigma),\tau}= \begin{cases}
            s_i(\sigma) & \text{if } \sigma=\tau,\\
            0 & \text{otherwise.}
        \end{cases}
    \]
    It is easy to check, using the fact that $\sum_{i:\sigma\in \mi_i} s_i(\sigma)^2=1$ for all $\sigma\in \mi$, that $S^{\T} S=I_{|\mi|}$. On the other hand, 
    note that $S^{\T}\left(\bigoplus_{i=1}^m M_i\right)S= M$.
    Therefore, by Theorem \ref{thm:cauchy_plus}, we have, for all $1\leq k \leq |\mi|$,
    \[
     \eigsmallest{k}{M} \geq  
     \eigsmallest{k}{\bigoplus_{i=1}^m M_i},
    \]
    and
    \[
     \eiglargest{k}{M} \leq  
      \eiglargest{k}{\bigoplus_{i=1}^m M_i}.
    \]
\end{proof}

\section{Garland's method via interlacing}\label{sec:main}

We proceed to prove our main results, Theorem \ref{thm:garland_plus} and Corollary \ref{cor:garland_plus}. First, we state and prove generalized versions of these results (Theorem \ref{thm:garland_main} and Corollary \ref{cor:main_garland}), which hold for a simplicial complex with an arbitrary weight function. In Section \ref{sec:main:special}, we show how they imply Theorem \ref{thm:garland_plus} and Corollary \ref{cor:garland_plus}, as well as their unweighted analogues.

Let $X$ be a simplicial complex on vertex set $V$, and let $w: X\to \Rea_{>0}$ be a weight function on $X$. Let $<$ be an arbitrary linear order on $V$. Recall that for $\sigma\in X$, we denote $N_X(\sigma)=\{v\in V\setminus\sigma:\, \sigma\cup\{v\}\in X\}$. For $0\le k\le \dim(X)$, let 
\[
    \wmaxdeg{k}{X}{w}= \max_{\sigma\in X(k)} \sum_{v\in N_X(\sigma)} \frac{w(\sigma\cup \{v\})}{w(\sigma)}.
\]
For $\eta\in X$, we define a weight function $w_{\eta}: \lk(X,\eta)\to \Rea_{>0}$ by
\[
    w_{\eta}(\sigma)= w(\sigma\cup\eta)
\]
for all $\sigma\in \lk(X,\eta)$. 
Our main goal in this section is to prove the following two results.

\begin{theorem}\label{thm:garland_main}
    Let $X$ be a simplicial complex, and let $w:X\to\Rea_{>0}$. Let $0\le \ell<k\le \dim(X)$. Then, for all $1\le i\le |X(k)|$, 
    \[
       (k-\ell)\cdot \eigsmallest{i}{\lapw{k}{X}{w}{}} \ge (k+1)\cdot \eigsmallest{i}{\bigoplus_{\substack{\eta\in X(\ell):\\\dim(\lk(X,\eta))\ge k-\ell-1}} \lapw{k-\ell-1}{\lk(X,\eta)}{w_{\eta}}{}} - (\ell+1)\wmaxdeg{k}{X}{w}.
    \]
\end{theorem}

\begin{corollary}\label{cor:main_garland}
    Let $X$ be a simplicial complex, and let $w:X\to\Rea_{>0}$. Then, for all $0\le \ell<k\le \dim(X)$,
    \[
        \dim(\tilde{H}_k(X;\Rea)) \le \sum_{\substack{\eta\in X(\ell):\\ \dim(\lk(X,\eta))\ge k-\ell-1}} \left| \left\{ \lambda\in {\rm Spec}(\lapw{k-\ell-1}{\lk(X,\eta)}{w_{\eta}}{}):\, \lambda\le \frac{\ell+1}{k+1}\cdot \wmaxdeg{k}{X}{w}\right\}\right|.
    \]
\end{corollary}

Let $X$ be a simplicial complex on vertex set $V$, and let $0\le \ell<k\le \dim(X)$. For $\eta\in X(\ell)$, let 
\[
\Xv{X}{\eta}{k}= \{\sigma\in X(k):\, \eta\subset\sigma\}.
\]
For $\eta\in X(\ell)$ and $\sigma\in \Xv{X}{\eta}{k}$, let
\[
s_{\eta}(\sigma)=\binom{k+1}{\ell+1}^{-1/2}\cdot (-1)^{|\{(i,j):\, i\in \sigma\setminus \eta,\, j\in \eta,\, i>j\}|}.
\]
Note that, for all $\sigma\in X(k)$, since $\left|\{\eta\in X(\ell):\, \sigma\in \Xv{X}{\eta}{k}\}\right|= \binom{k+1}{\ell+1}$, we have 
\[
\sum_{\eta\in X(\ell):\, \sigma\in \Xv{X}{\eta}{k}} s_{\eta}(\sigma)^2=1.
\]

Let $R_k(X;w)\in \matsi{X(k)}$ be the diagonal matrix defined by
\begin{equation}\label{eq:error_term_def}
    R_k(X;w)_{\sigma,\sigma}= \sum_{v\in N_X(\sigma)}\frac{w(\sigma\cup\{v\})}{w(\sigma)}
\end{equation}
for all $\sigma\in X(k)$. Note that $\eiglargest{1}{R_k(X;w)}=\wmaxdeg{k}{X}{w}$. 

 Assume $\dim(\lk(X,\eta))\ge k-\ell-1$. Note that the map $\sigma\mapsto\sigma\setminus\eta$ is a bijection from $\Xv{X}{\eta}{k}$ to $\lk(X,\eta)(k-\ell-1)$.
  Hence, we may define matrices $\simlapws{k-\ell-1}{\lk(X,\eta)}{w_{\eta}}{+}$, $\simlapws{k-\ell-1}{\lk(X,\eta)}{w_{\eta}}{-}$, and $\simlapws{k-\ell-1}{\lk(X,\eta)}{w_{\eta}}{}$ in $\matsi{\Xv{X}{\eta}{k}}$ by
\[
\simlapws{k-\ell-1}{\lk(X,\eta)}{w_{\eta}}{+}_{\sigma,\tau}=\simlapw{k-\ell-1}{\lk(X,\eta)}{w_{\eta}}{+}_{\sigma\setminus\eta,\tau\setminus\eta},
\]
\[
\simlapws{k-\ell-1}{\lk(X,\eta)}{w_{\eta}}{-}_{\sigma,\tau}=\simlapw{k-\ell-1}{\lk(X,\eta)}{w_{\eta}}{-}_{\sigma\setminus\eta,\tau\setminus\eta},
\]
and
\[
\simlapws{k-\ell-1}{\lk(X,\eta)}{w_{\eta}}{}_{\sigma,\tau}=\simlapw{k-\ell-1}{\lk(X,\eta)}{w_{\eta}}{}_{\sigma\setminus\eta,\tau\setminus\eta},
\]
for all $\sigma,\tau\in \Xv{X}{\eta}{k}$. 

Note that $\simlapws{k-\ell-1}{\lk(X,\eta)}{w_{\eta}}{+}$, $\simlapws{k-\ell-1}{\lk(X,\eta)}{w_{\eta}}{-}$, and $\simlapws{k-\ell-1}{\lk(X,\eta)}{w_{\eta}}{}$ are equal, as matrices, to $\simlapw{k-\ell-1}{\lk(X,\eta)}{w_{\eta}}{+}$, $\simlapw{k-\ell-1}{\lk(X,\eta)}{w_{\eta}}{-}$, and $\simlapw{k-\ell-1}{\lk(X,\eta)}{w_{\eta}}{}$, respectively, and differ from them only in their indexing convention.

The proof of Theorem \ref{thm:garland_main} relies on the two following key identities.

\begin{proposition}\label{prop:lplus}
    Let $X$ be a simplicial complex, and let $w:X\to \Rea_{>0}$.  Let $0\le \ell<k\le \dim(X)$. Then, for all $\sigma,\tau\in X(k)$,
    \[  
         \sum_{\substack{\eta\in X(\ell):\\ \sigma,\tau\in \Xv{X}{\eta}{k}} }s_{\eta}(\sigma)s_{\eta}(\tau)\cdot (k+1)\cdot \simlapws{k-\ell-1}{\lk(X,\eta)}{w_{\eta}}{+}_{\sigma,\tau}=  (k-\ell)\cdot \simlapw{k}{X}{w}{+}_{\sigma,\tau} + (\ell+1)\cdot R_k(X;w)_{\sigma,\tau}. 
    \]
\end{proposition}

\begin{proposition}
    \label{prop:lminus}
      Let $X$ be a simplicial complex, and let $w:X\to \Rea_{>0}$.  Let $0\le \ell<k\le \dim(X)$. Then, for all $\sigma,\tau\in X(k)$,
    \[  
         \sum_{\substack{\eta\in X(\ell):\\ \sigma,\tau\in \Xv{X}{\eta}{k}} }s_{\eta}(\sigma)s_{\eta}(\tau)\cdot (k+1) \cdot \simlapws{k-\ell-1}{\lk(X,\eta)}{w_{\eta}}{-}_{\sigma,\tau}=  (k-\ell)\cdot \simlapw{k}{X}{w}{-}_{\sigma,\tau}.
    \]
\end{proposition}

First, let us prove Theorem \ref{thm:garland_main} and Corollary \ref{cor:main_garland}, assuming Propositions \ref{prop:lplus} and \ref{prop:lminus}.

\begin{proof}[Proof of Theorem \ref{thm:garland_main}]
    Let $M=\frac{1}{k+1}\left((k-\ell) \simlapw{k}{X}{w}{} +(\ell+1)R_k(X;w)\right)$. 
    By Propositions \ref{prop:lplus} and \ref{prop:lminus}, we have
    \[  
         M_{\sigma,\tau}=\sum_{\substack{\eta\in X(\ell):\\ \sigma,\tau\in \Xv{X}{\eta}{k}} }s_{\eta}(\sigma)s_{\eta}(\tau)\simlapws{k-\ell-1}{\lk(X,\eta)}{w_{\eta}}{}_{\sigma,\tau}
    \]
    for all $\sigma,\tau\in X(k)$. Let $1\le i\le |X(k)|$. By Lemma \ref{lemma:abstract_garland},
    \begin{align*}
        \eigsmallest{i}{M}&\ge \eigsmallest{i}{\bigoplus_{\substack{\eta\in X(\ell):\\\dim(\lk(X,\eta))\ge k-\ell-1}} \simlapws{k-\ell-1}{\lk(X,\eta)}{w_{\eta}}{}}
       \\ &= \eigsmallest{i}{\bigoplus_{\substack{\eta\in X(\ell):\\\dim(\lk(X,\eta))\ge k-\ell-1}} \lapw{k-\ell-1}{\lk(X,\eta)}{w_{\eta}}{}}.
    \end{align*}
    On the other hand, by Lemma \ref{lemma:weyl}, we have
\begin{align*}
(k-\ell)\cdot\eigsmallest{i}{\lapw{k}{X}{w}{}}= (k-\ell)\cdot\eigsmallest{i}{\simlapw{k}{X}{w}{}} &\ge (k+1)\cdot \eigsmallest{i}{M} +(\ell+1)\cdot\eigsmallest{1}{-R_k(X;w)} \\&=  (k+1)\cdot \eigsmallest{i}{M} -(\ell+1)\cdot\Delta_k(X;w).
\end{align*}
    Hence,
    \[
       (k-\ell)\cdot \eigsmallest{i}{\lapw{k}{X}{w}{}} \ge (k+1)\cdot \eigsmallest{i}{\bigoplus_{\substack{\eta\in X(\ell):\\\dim(\lk(X,\eta))\ge k-\ell-1}} \lapw{k-\ell-1}{\lk(X,\eta)}{w_{\eta}}{}} - (\ell+1)\cdot \Delta_k(X;w).
    \]

\end{proof}

\begin{proof}[Proof of Corollary \ref{cor:main_garland}]
    Let \[m=\sum_{\substack{\eta\in X(\ell):\\\dim(\lk(X,\eta))\ge k-\ell-1}} \left| \left\{ \lambda\in {\rm Spec}(\lapw{k-\ell-1}{\lk(X,\eta)}{w_{\eta}}{}):\, \lambda\le \frac{\ell+1}{k+1}\cdot \wmaxdeg{k}{X}{w}\right\}\right|. \]
    If $m\ge |X(k)|$, then $\dim(\tilde{H}_k(X;\Rea))\le m$ trivially. Otherwise, by the definition of $m$, we have
    \[
       \eigsmallest{m+1}{\bigoplus_{\substack{\eta\in X(\ell):\\\dim(\lk(X,\eta))\ge k-\ell-1}} \lapw{k-\ell-1}{\lk(X,\eta)}{w_{\eta}}{}}>\frac{\ell+1}{k+1}\cdot \Delta_k(X;w).
    \]
    Therefore, by Theorem \ref{thm:garland_main},
     \[
       (k-\ell)\cdot\eigsmallest{m+1}{\lapw{k}{X}{w}{}} > (k+1)\cdot \frac{\ell+1}{k+1}\cdot \Delta_k(X;w) - (\ell+1)\cdot \Delta_k(X;w)= 0.
    \]
    So, $\eigsmallest{m+1}{\lapw{k}{X}{w}{}}>0$.
    Thus, by Lemma \ref{lemma:eckmann}, $\dim(\tilde{H}_k(X;\Rea))\le m$, as wanted.
\end{proof}

For the proofs of Propositions \ref{prop:lplus} and \ref{prop:lminus}, we will need the following lemmas.

\begin{lemma}\label{lemma:sign1}
    Let $X$ be a simplicial complex, and let $0\le \ell<k\le \dim(X)$. 
    Let $\eta\in X(\ell)$ and $\tau\in \Xv{X}{\eta}{k}$. Let $u\in \tau\setminus \eta$, and let $\sigma=\tau\setminus\{u\}$. Then,
    \[
        s_{\eta}(\tau)\cdot (\tau\setminus\eta:\sigma\setminus\eta)=\sqrt{\frac{k-\ell}{k+1}} \cdot s_{\eta}(\sigma)\cdot (\tau:\sigma). 
    \]
\end{lemma}
\begin{proof}
Note that
\begin{align*}
|\{(i,j):\, & i\in \tau\setminus \eta,\, j\in \eta,\, i>j\}| +|\{j\in \tau\setminus \eta:\, j<u\}|
\\
&= |\{(i,j):\, i\in \sigma\setminus \eta,\, j\in \eta,\, i>j\}| +|\{j\in \eta:\, u>j\}|+|\{j\in \tau\setminus \eta:\, j<u\}| 
\\
&= |\{(i,j):\, i\in \sigma\setminus \eta,\, j\in \eta,\, i>j\}| +|\{j\in \tau:\, j<u\}| .
\end{align*}
Therefore, 
      \begin{align*}
     s_{\eta}(\tau)\cdot (\tau\setminus\eta:\sigma\setminus\eta) &= \binom{k+1}{\ell+1}^{-1/2}\cdot(-1)^{|\{(i,j):\, i\in \tau\setminus \eta,\, j\in \eta,\,  i>j\}|}\cdot (-1)^{|\{j\in \tau\setminus\eta:\, j<u\}|}
     \\
    &=\sqrt{\frac{k-\ell}{k+1}}\cdot \binom{k}{\ell+1}^{-1/2}\cdot  (-1)^{|\{(i,j):\, i\in\sigma\setminus \eta,\, j\in \eta,\, i>j\}|}\cdot (-1)^{|\{j\in \tau:\, j<u\}|}
    \\
    &= \sqrt{\frac{k-\ell}{k+1}} \cdot s_{\eta}(\sigma) \cdot (\tau:\sigma).
   \end{align*}
\end{proof}

\begin{lemma}\label{lemma:sign2}
    Let $X$ be a simplicial complex, and let $0\le \ell<k\le \dim(X)$. 
    Let $\eta\in X(\ell)$ and $\sigma,\tau \in \Xv{X}{\eta}{k}$ such that $|\sigma\cap \tau|=k$. Then,
    \[
        (\sigma:\sigma\cap \tau)(\tau:\sigma\cap \tau)=  \binom{k+1}{\ell+1}\cdot s_{\eta}(\sigma)s_{\eta}(\tau)\cdot (\sigma\setminus\eta: (\sigma\cap \tau)\setminus\eta)\cdot (\tau\setminus\eta: (\sigma\cap \tau)\setminus\eta).
    \]
\end{lemma}
\begin{proof}
By Lemma \ref{lemma:sign1}, we have
\begin{align*}
    s_{\eta}(\sigma)s_{\eta}(\tau) &\cdot (\sigma\setminus\eta: (\sigma\cap \tau)\setminus\eta)\cdot (\tau\setminus\eta: (\sigma\cap \tau)\setminus\eta)
    = \frac{k-\ell}{k+1}\cdot s_{\eta}(\sigma\cap\tau)^2 \cdot (\sigma:\sigma\cap \tau)(\tau:\sigma\cap \tau)
    \\
    &=\frac{k-\ell}{k+1}\cdot \binom{k}{\ell+1}^{-1}\cdot (\sigma:\sigma\cap \tau)(\tau:\sigma\cap \tau)
    =\binom{k+1}{\ell+1}^{-1}\cdot (\sigma:\sigma\cap \tau)(\tau:\sigma\cap \tau),
\end{align*}
as required.
\end{proof}

For a finite set $A$ and $n\ge 0$, we denote by $\binom{A}{n}$ the family of all subsets of $A$ of size $n$.

\begin{proof}[Proof of Proposition \ref{prop:lplus}]
    For $\eta\in X(\ell)$ such that $\dim(\lk(X,\eta))\ge k-\ell-1$, we denote $L_{\eta}^+=\simlapws{k-\ell-1}{\lk(X,\eta)}{w_{\eta}}{+}$. Note that, for $\sigma\in \Xv{X}{\eta}{k}$, we have $N_{\lk(X,\eta)}(\sigma\setminus\eta)=N_X(\sigma)$, $s_{\eta}(\sigma)^2=\binom{k+1}{\ell+1}^{-1}$, and $w_{\eta}(\sigma\setminus \eta)=w(\sigma)$.
    In addition, note that, for $\sigma,\tau\in \Xv{X}{\eta}{k}$, we have $|\sigma\cap \tau|=k$ if and only if $|(\sigma\setminus\eta)\cap (\tau\setminus\eta)|=k-\ell-1$, and that  $\sigma\cup\tau\in X(k+1)$ if and only if $(\sigma\cup \tau)\setminus \eta\in \lk(X,\eta)(k-\ell)$. Moreover, by Lemma \ref{lemma:sign2}, 
       \[
       \binom{k+1}{\ell+1} s_{\eta}(\sigma)s_{\eta}(\tau)(\sigma\setminus\eta:\, (\sigma\cap\tau)\setminus\eta)(\tau\setminus\eta:\, (\sigma\cap\tau)\setminus\eta)= (\sigma:\sigma\cap \tau)(\tau:\sigma\cap \tau).
   \]
    Therefore, by \eqref{eq:upper_sim_laplacian_def},
    \begin{align*}
       &\binom{k+1}{\ell+1}\cdot s_{\eta}(\sigma)s_{\eta}(\tau)\cdot (L_{\eta}^+)_{\sigma,\tau} \\&= \begin{cases}
             \sum_{u\in N_X(\sigma)} \frac{w(\sigma\cup\{u\})}{w(\sigma)} & \text{if } \sigma=\tau,\\
            - (\sigma:\sigma\cap \tau)(\tau:\sigma\cap \tau) \frac{w(\sigma\cup \tau)}{\sqrt{w(\sigma)w(\tau)}} & \text{if } |\sigma\cap \tau|=k,\, \sigma\cup \tau\in X(k+1),\\
            0 & \text{otherwise,}
        \end{cases}
    \end{align*}
    for all $\sigma,\tau\in\Xv{X}{\eta}{k}$.

Note that, for every $\sigma\in X(k)$, $|\{\eta\in X(\ell):\, \sigma\in \Xv{X}{\eta}{k}\}|=\left|\binom{\sigma}{\ell+1}\right|=\binom{k+1}{\ell+1}$. Similarly, for $\sigma,\tau\in X(k)$ such that $|\sigma\cap \tau|=k$, $|\{\eta\in X(\ell):\, \sigma,\tau\in \Xv{X}{\eta}{k}\}|=\left|\binom{\sigma\cap \tau}{\ell+1}\right|=\binom{k}{\ell+1}$. Therefore,
    \begin{align*}
        &\binom{k+1}{\ell+1}\cdot \sum_{\substack{\eta\in X(\ell):\\ \sigma,\tau\in \Xv{X}{\eta}{k}} } s_{\eta}(\sigma)s_{\eta}(\tau)\cdot (L_{\eta}^+)_{\sigma,\tau} \\ &= \begin{cases}
            \binom{k+1}{\ell+1}\cdot \sum_{u\in N_X(\sigma)} \frac{w(\sigma\cup\{u\})}{w(\sigma)} & \text{if } \sigma=\tau,\\
            -\binom{k}{\ell+1}\cdot (\sigma:\sigma\cap \tau)(\tau:\sigma\cap \tau) \frac{w(\sigma\cup \tau)}{\sqrt{w(\sigma)w(\tau)}} & \text{if } |\sigma\cap \tau|=k,\, \sigma\cup \tau\in X(k+1),\\
            0 & \text{otherwise,}
        \end{cases}
    \end{align*}
    for all $\sigma,\tau\in X(k)$. By \eqref{eq:upper_sim_laplacian_def} and \eqref{eq:error_term_def}, we obtain, for all $\sigma,\tau\in X(k)$, 
    \[
       \binom{k+1}{\ell+1}\cdot \sum_{\substack{\eta\in X(\ell):\\ \sigma,\tau\in \Xv{X}{\eta}{k}} } s_{\eta}(\sigma)s_{\eta}(\tau)\cdot (L_{\eta}^+)_{\sigma,\tau}=  \binom{k}{\ell+1} \cdot \simlapw{k}{X}{w}{+}_{\sigma,\tau} + \binom{k}{\ell}R_k(X;w)_{\sigma,\tau}.
    \]
    Multiplying both sides of the equation by $(\ell+1)\cdot \binom{k}{\ell}^{-1}$, we obtain
    \[
       \sum_{\substack{\eta\in X(\ell):\\ \sigma,\tau\in \Xv{X}{\eta}{k}} } s_{\eta}(\sigma)s_{\eta}(\tau)\cdot (k+1)\cdot (L_{\eta}^+)_{\sigma,\tau}=  (k-\ell) \cdot \simlapw{k}{X}{w}{+}_{\sigma,\tau} + (\ell+1)\cdot R_k(X;w)_{\sigma,\tau},
    \]
    for all $\sigma,\tau\in X(k)$.

\end{proof}

\begin{proof}[Proof of Proposition \ref{prop:lminus}]
    
 For $\eta\in X(\ell)$ such that $\dim(\lk(X,\eta))\ge k-\ell-1$, we denote $L_{\eta}^{-}=\simlapws{k-\ell-1}{\lk(X,\eta)}{w_{\eta}}{-}$. 
   Note that $w_{\eta}(\sigma\setminus\eta)=w(\sigma)$ and $s_{\eta}(\sigma)^2= \binom{k+1}{\ell+1}^{-1}$ for all $\sigma\in \Xv{X}{\eta}{k}$. Moreover, for all $\sigma,\tau\in \Xv{X}{\eta}{k}$, we have $|\sigma\cap \tau|=k$ if and only if $|(\sigma\setminus\eta)\cap(\tau\setminus\eta)|=k-\ell-1$. By Lemma \ref{lemma:sign2}, we have in this case
   \[
       \binom{k+1}{\ell+1} s_{\eta}(\sigma)s_{\eta}(\tau)(\sigma\setminus\eta:\, (\sigma\cap\tau)\setminus\eta)(\tau\setminus\eta:\, (\sigma\cap\tau)\setminus\eta)= (\sigma:\sigma\cap \tau)(\tau:\sigma\cap \tau).
   \]
   Hence, by \eqref{eq:lower_sim_laplacian_def}, we obtain
    \[
      \binom{k+1}{\ell+1} s_{\eta}(\sigma)s_{\eta}(\tau) (L_{\eta}^{-})_{\sigma,\tau}= \begin{cases}
            \sum_{v\in \sigma\setminus\eta} \frac{w(\sigma)}{w(\sigma\setminus\{v\})} & \text{if } \sigma=\tau,\\
            (\sigma:\sigma\cap \tau)(\tau:\sigma\cap \tau)\frac{\sqrt{w(\sigma)w(\tau)}}{w(\sigma\cap \tau)} & \text{if } |\sigma\cap \tau|=k,\\
            0 & \text{otherwise,}
        \end{cases}
    \]
    for all $\sigma,\tau\in X_{\eta}(k)$.
    Note that, for every $\sigma\in X(k)$, $\{\eta\in X(\ell):\, \sigma\in \Xv{X}{\eta}{k}\}=\binom{\sigma}{\ell+1}$, and
    \[
        \sum_{\eta\in \binom{\sigma}{\ell+1}} \sum_{v\in \sigma\setminus\eta} \frac{w(\sigma)}{w(\sigma\setminus\{v\})}= \binom{k}{\ell+1}\cdot \sum_{v\in \sigma} \frac{w(\sigma)}{w(\sigma\setminus\{v\})}.
    \]
    Furthermore, for $\sigma,\tau\in X(k)$ such that $|\sigma\cap \tau|=k$, $|\{\eta\in X(\ell):\, \sigma,\tau\in \Xv{X}{\eta}{k}\}|=\binom{k}{\ell+1}$. 
    Therefore,
    \[
      \binom{k+1}{\ell+1}\cdot \sum_{\substack{\eta\in X(\ell):\\ \sigma,\tau\in \Xv{X}{\eta}{k}}} s_{\eta}(\sigma)s_{\eta}(\tau) (L_{\eta}^{-})_{\sigma,\tau} = \begin{cases}
             \binom{k}{\ell+1}\cdot \sum_{v\in \sigma} \frac{w(\sigma)}{w(\sigma\setminus\{v\})} & \text{if } \sigma=\tau,\\
            \binom{k}{\ell+1}\cdot (\sigma:\sigma\cap \tau)(\tau:\sigma\cap \tau)\frac{\sqrt{w(\sigma)w(\tau)}}{w(\sigma\cap \tau)} & \text{if } |\sigma\cap \tau|=k,\\
            0 & \text{otherwise,}
            \end{cases}
    \]
    for all $\sigma,\tau\in X(k)$.
    Hence, by \eqref{eq:lower_sim_laplacian_def},
    \[
        \binom{k+1}{\ell+1}\cdot \sum_{\substack{\eta\in X(\ell):\\ \sigma,\tau\in \Xv{X}{\eta}{k}}} s_{\eta}(\sigma)s_{\eta}(\tau) (L_{\eta}^{-})_{\sigma,\tau}   =  \binom{k}{\ell+1}\cdot \simlapw{k}{X}{w}{-}_{\sigma,\tau},
    \]
    for all $\sigma,\tau\in X(k)$. Multiplying both sides of the equation by $(\ell+1)\binom{k}{\ell}^{-1}$, we obtain
    \[
         \sum_{\substack{\eta\in X(\ell):\\ \sigma,\tau\in \Xv{X}{\eta}{k}}} s_{\eta}(\sigma)s_{\eta}(\tau) \cdot (k+1)\cdot (L_{\eta}^{-})_{\sigma,\tau}   =  (k-\ell)\cdot \simlapw{k}{X}{w}{-}_{\sigma,\tau}, 
    \]
    as wanted.
\end{proof}

\subsection{Upper bounds}\label{sec:main:upper}

Our arguments can also be used to prove upper bounds on the eigenvalues of Laplacian operators on simplicial complexes, similar to the lower bounds in Theorem \ref{thm:garland_main}. We state the next result in terms of the upper Laplacian, but an analogous bound for the total Laplacian can be proven similarly. 

For a simplicial complex $X$, a weight function $w: X\to \Rea_{>0}$, and $k\ge 0$, let 
\[
\wmindeg{k}{X}{w}=\min_{\sigma\in X(k)} \sum_{v\in N_X(\sigma)} \frac{w(\sigma\cup\{v\})}{w(\sigma)}.
\]
Note that $\wmindeg{k}{X}{w}= \eigsmallest{1}{R_k(X;w)}$.

\begin{proposition}\label{prop:garland_upper}
     Let $X$ be a simplicial complex, and let $w:X\to\Rea_{>0}$. Let $0\le \ell<k\le \dim(X)-1$. Then, for all $1\le i\le |X(k)|$, 
    \[
       (k-\ell)\cdot \eiglargest{i}{\lapw{k}{X}{w}{+}} \le (k+1)\cdot \eiglargest{i}{\bigoplus_{\substack{\eta\in X(\ell):\\\dim(\lk(X,\eta))\ge k-\ell-1}} \lapw{k-\ell-1}{\lk(X,\eta)}{w_{\eta}}{+}} - (\ell+1)\wmindeg{k}{X}{w}.
    \]
\end{proposition}

The proof is essentially the same as that of Theorem \ref{thm:garland_main}, so we omit the details. The case $i=1$ of Proposition \ref{prop:garland_upper} was observed (in the special case of weighted pure complexes) by Papikian in \cite{papikian2016garland} and by Gundert and Wagner in \cite{gundert2016eigenvalues}.

\subsection{Application to different choices of weight functions}\label{sec:main:special}

\subsubsection{The unweighted case}\label{sec:main:unweighted}

Let $X$ be a simplicial complex, and let $0\le k\le \dim(X)$. Let $w: X\to \Rea_{>0}$ be the constant function $w\equiv 1$. We denote $L_k^+(X)=\lapw{k}{X}{w}{+}$ and $L_k(X)=\lapw{k}{X}{w}{}$. Note that, for $0\le \ell<k$ and $\eta\in X(\ell)$, $w_{\eta}\equiv 1$, so $\lapw{k-\ell-1}{\lk(X,\eta)}{w_{\eta}}{+}=L_{k-\ell-1}^+(\lk(X,\eta))$ and $\lapw{k-\ell-1}{\lk(X,\eta)}{w_{\eta}}{}=L_{k-\ell-1}(\lk(X,\eta))$.  
For $\sigma\in X(k)$, let $\deg(\sigma)= |\{\tau\in X(k+1):\, \sigma\subset \tau\}|$. Let $\Delta_k(X)= \max_{\sigma\in X(k)} \deg(\sigma)$, and $\delta_k(X)=\min_{\sigma\in X(k)} \deg(\sigma)$. Note that $\wmaxdeg{k}{X}{w}=\Delta_k(X)$ and $\wmindeg{k}{X}{w}=\delta_k(X)$. Therefore, by Theorem \ref{thm:garland_main} and Proposition \ref{prop:garland_upper}, for all $1\le i\le |X(k)|$, 
    \begin{equation}\label{eq:unweighted_garland}
       (k-\ell)\cdot \eigsmallest{i}{L_k(X)} \ge (k+1)\cdot \eigsmallest{i}{\bigoplus_{\substack{\eta\in X(\ell):\\\dim(\lk(X,\eta))\ge k-\ell-1}} L_{k-\ell-1}(\lk(X,\eta))} - (\ell+1)\Delta_k(X),
    \end{equation}
    and
    \[
       (k-\ell)\cdot \eiglargest{i}{L_k^+(X)} \le (k+1)\cdot \eiglargest{i}{\bigoplus_{\substack{\eta\in X(\ell):\\\dim(\lk(X,\eta))\ge k-\ell-1}} L_{k-\ell-1}^+(\lk(X,\eta))} -(\ell+1)\delta_k(X). 
    \]
Moreover, by Corollary \ref{cor:main_garland}, 
\[
   \dim(\tilde{H}_k(X;\Rea)) \le \sum_{\substack{\eta\in X(\ell):\\\dim(\lk(X,\eta))\ge k-\ell-1}} \left| \left\{ \lambda\in {\rm Spec}(L_{k-\ell-1}(\lk(X,\eta)):\, \lambda\le \frac{\ell+1}{k+1}\cdot \Delta_k(X)\right\}\right|.
\]
Let us note that the $i=1$ case of \eqref{eq:unweighted_garland} was first observed by Parzanchevski, Rosenthal, and Tessler in \cite{parzanchevski2016isoperimetric}.

\subsubsection{Weighted pure simplicial complexes}\label{sec:main:weighted}

Recall that a simplicial complex $X$ is called a pure $d$-dimensional complex if all its maximal faces have dimension $d$.  Let $X$ be a pure $d$-dimensional simplicial complex, and let $0\le k\le d$. Recall that we denote $\tilde{L}_k^+(X)=\lapw{k}{X}{w}{+}$ and $\tilde{L}_k(X)=\lapw{k}{X}{w}{}$, where $w:X\to \Rea_{>0}$ is defined by
\[
    w(\sigma)= |\{\tau\in X(d):\, \sigma\subset \tau\}|.
\]
That is, $w(\sigma)$ is the number of maximal faces of $X$ containing $\sigma$. Note that, since $X$ is pure, $w(\sigma)>0$ for all $\sigma\in X$. Moreover, for $0\le \ell<k$ and $\eta\in X(\ell)$, $\lk(X,\eta)$ is a pure $(d-\ell-1)$-dimensional complex, and $w_{\eta}(\sigma)=|\{\tau\in X(d):\, \sigma\cup\eta\subset\tau\}|=|\{\tau'\in \lk(X,\eta)(d-\ell-1):\, \sigma\subset\tau'\}|$ for all $\sigma\in \lk(X,\eta)$. Hence, $\lapw{k-\ell-1}{\lk(X,\eta)}{w_{\eta}}{+}=\tilde{L}_{k-\ell-1}^+(\lk(X,\eta))$ and
$\lapw{k-\ell-1}{\lk(X,\eta)}{w_{\eta}}{}=\tilde{L}_{k-\ell-1}(\lk(X,\eta))$. In addition, we have, for all $k\ge 0$ and $\sigma\in X(k)$,
\[
\sum_{v\in N_X(\sigma)} \frac{w(\sigma\cup\{v\})}{w(\sigma)}=\frac{1}{w(\sigma)}\cdot \sum_{\substack{\tau\in X(d):\\ \sigma\subset\tau}}\,\sum_{v\in \tau\setminus\sigma}1= \frac{1}{w(\sigma)}\cdot w(\sigma)\cdot (d-k)= d-k.
\]
Thus, $\Delta_k(X;w)=\delta_k(X;w)=d-k$.
 Therefore, by Theorem \ref{thm:garland_main}, for all $0\le \ell<k\le d$ and $1\le i\le |X(k)|$, 
    \[
       (k-\ell)\cdot \eigsmallest{i}{\tilde{L}_k(X)} \ge (k+1)\cdot \eigsmallest{i}{\bigoplus_{\eta\in X(\ell)} \tilde{L}_{k-\ell-1}(\lk(X,\eta))} - (\ell+1)(d-k),
    \]
    thus proving Theorem \ref{thm:garland_plus}. Similarly, by Proposition \ref{prop:garland_upper}, 
    \[
       (k-\ell)\cdot \eiglargest{i}{\tilde{L}_k^+(X)} \le (k+1)\cdot \eiglargest{i}{\bigoplus_{\eta\in X(\ell)} \tilde{L}_{k-\ell-1}^+(\lk(X,\eta))} -(\ell+1)(d-k). 
    \]
Finally, by Corollary \ref{cor:main_garland}, 
\[
   \dim(\tilde{H}_k(X;\Rea)) \le \sum_{\eta\in X(\ell)} \left| \left\{ \lambda\in {\rm Spec}(\tilde{L}_{k-\ell-1}(\lk(X,\eta))):\, \lambda\le  \frac{(\ell+1)(d-k)}{k+1}\right\}\right|,
\]
proving Corollary \ref{cor:garland_plus}.

\bibliographystyle{abbrv}
\bibliography{main}

\end{document}